\documentclass[10pt,a4paper,reqno]{article}
\usepackage{float}
\usepackage{authblk}
\usepackage{graphicx}
\usepackage{titlesec}
\usepackage{enumerate}   
\titleformat{\section}{\normalfont\fontsize{12.5}{10}\bfseries}{\thesection}{0.5em}{}
\titleformat{\subsection}{\normalfont\fontsize{11.5}{17}\bfseries}{\thesubsection}{0.5em}{}
\usepackage{amsmath}
\allowdisplaybreaks
\usepackage[margin=2.8cm,includehead]{geometry}
\usepackage[table]{xcolor}
\usepackage{comment}
\graphicspath{}
\usepackage[linktocpage=true,colorlinks,citecolor=blue,linkcolor=blue,urlcolor=blue, pagebackref]{hyperref}
\usepackage{amssymb}
\usepackage{cite}
\usepackage{amsmath}
\usepackage{latexsym}
\usepackage{amscd}
\usepackage{amsthm}
\usepackage{mathrsfs}
\usepackage{url}
\usepackage{amsmath,mathtools}
\usepackage{cleveref}
\usepackage{hyperref}
\usepackage[alphabetic,backrefs]{amsrefs}
\usepackage{lipsum} 
\usepackage{mathtools}
\usepackage{tikzpagenodes}
\usepackage{faktor}
\usetikzlibrary{tikzmark}
\usepackage{color}
\usepackage[T1]{fontenc}
\usepackage{slashed}

\newtheorem{thm}{Theorem}[section]
\newtheorem{corr}[thm]{Corollary}
\newtheorem{lemma}[thm]{Lemma}
\newtheorem{prop}[thm]{Proposition}

\theoremstyle{definition}
\newtheorem{defn}[thm]{Definition}
\newtheorem{rem}[thm]{Remark}

\numberwithin{equation}{section}
\def\R{\mathbb R}
\def\Re{\rm{Re} }
\def\Im{\rm{Im}}
\def\SU{\rm{SU}}

\def\Z{\mathbb Z}
\def\so{\mathfrak{so}}

\def\del{\nabla}
\def\G2{\mathrm{G}_2}
\def\g2{\varphi}

\def\su{\mathfrak{su}}
\def\SO{\rm{SO}}
\def\id{\textup{id}}

\def\lieg2{\mathfrak{g}_2}

\DeclareMathOperator\vol{vol}

\DeclareMathOperator\tr{tr}

\DeclareFontFamily{U}{MnSymbolC}{}
\DeclareSymbolFont{MnSyC}{U}{MnSymbolC}{m}{n}
\DeclareFontShape{U}{MnSymbolC}{m}{n}{
    <-6>  MnSymbolC5
   <6-7>  MnSymbolC6
   <7-8>  MnSymbolC7
   <8-9>  MnSymbolC8
   <9-10> MnSymbolC9
  <10-12> MnSymbolC10
  <12->   MnSymbolC12}{}
\DeclareMathSymbol{\intprod}{\mathbin}{MnSyC}{'270}

\begin{document}

\title{Nearly half-flat \rm{SU}(3) structures on $S^3\times S^3$}
\author{Ragini Singhal}
\date{}
\maketitle

\textbf{Abstract}
We study the $\SU(3)$-structure induced on an oriented hypersurface of a 7-dimensional manifold with a nearly parallel $\G2$-structure. Such $\rm{SU}(3)$-structures are called \textit{nearly half-flat}. 
We characterise the left invariant nearly half-flat structures on $S^3\times S^3$.  This characterisation then helps us to systematically analyse nearly parallel $\G2$-structures on an interval times $ S^3\times S^3$. 

\tableofcontents{}
\section{Introduction}

A $\G2$-structure on a $7$-dimensional manifold $M$ is defined by a non-degenerate 3-form $\g2$ that induces a metric $g_\g2$, a cross-product $\times_\g2$, an orientation $\vol_\g2$ and thus a Hodge star $*_\g2$ on $M$ (see \cite{bryantrmks}). The Riemannian manifold $M$ with a $\G2$-structure $\g2$ is called nearly $\G2$ if $\g2$ is a nearly parallel $\G2$-structure that is, for some $\lambda\neq 0$ 
\begin{align}\label{eq:ng2-intro}
    d\g2&=\lambda *_\g2 \g2.
\end{align}
They were described as manifolds with weak holonomy $\G2$ by Gray in \cite{Gray1971} and are positive Einstein \cite{friedkath}. The cone over a nearly parallel $\G2$-manifold has holonomy contained in $\textup{Spin}(7)$. Manifolds with nearly parallel $\G2$-structure have been studied for various reasons in mathematics as well as physics: in addition to the papers described in the introduction, see also \cites{BILAL2002112,BILAL2003343, AgrFried, deformg2,Gemmer2011YangMillsIO,Kawai2016SecondOrderDO,gonball,Ball-Madnick, dwivedi2020deformation,podesta_2021, singhal-instanton,freund-supergravity, Acharya:2003ii, BEHRNDT200499}.

So far the only known examples of nearly $\G2$-manifolds are homogeneous \cite{friedkath} or they come from 3-Sasakian geometry \cites{Galicki2,wilking-3sasaki}. Thus, in order to find new examples of nearly parallel $\G2$-manifolds one has to look for non-homogeneous examples. Finding an inhomogeneous Einstein metric coming from the nearly parallel $\G2$-structure is quite  challenging since the field equations for an appropriate metric ansatz are highly nonlinear partial differential equations. As a first step towards this goal, it would be natural to look for  cohomogeneity-one examples of nearly parallel $\G2$-manifolds, that is, nearly parallel $\G2$-structures that has a Lie group $G$ action preserving $\g2$ and the generic orbits have dimension $7-1=6$. These $6$-dimensional orbits then carry an invariant $\SU(3)$-structure defined by a 2-form $\omega$ and a $3$-form $\gamma$. In this article we describe the invariant $\SU(3)$-structure $(\omega,\gamma)$ on the six-dimensional oriented hypersurface $M^6$ of a nearly parallel $\G2$-manifold $N^7$. The condition that the $\G2$-structure $\g2=dt\wedge\omega+\gamma$ on $I\times M$ satisfies \eqref{eq:ng2-intro} imposes some conditions on the intrinsic torsion of the $\SU(3)$-structure $(\omega,\gamma)$ on $M$. The $\SU(3)$-structure $(\omega,\gamma)$ thus obtained is  known as \textit{nearly half-flat} and is defined by the condition
\begin{align*}
    d\gamma=\frac{\lambda}{2}\omega^2
\end{align*} where $\lambda$ is the non-zero constant in \eqref{eq:ng2-intro}.

The term nearly half-flat originates from the analogous situation where the $\SU(3)$-structure on the hypersurface of a torsion-free $\G2$-manifold is referred as \textit{half-flat}. Nearly half-flat $\SU(3)$-structures were first introduced in \cite{nearlyhypo-nhf} in the context of
evolution equations on six-manifolds $M$ leading to nearly parallel $\G2$-structures on the product of $M$ and an interval. In \cites{cvlt-liftnhf,liftingsu3tonhf,Fabian-thesis,Conti-embedding} the authors showed that one can construct nearly parallel $\G2$-structures by lifting certain nearly half-flat structures. This result is analogous to Hitchin's result in \cite{Hitchin-stableforms} that half-flat 
$\SU(3)$-structures on a six-dimensional manifold $M$ can be lifted to parallel $\G2$-structure on the product $M\times\R$ under certain conditions. In fact the nearly half-flat is a slight generalisation of the half-flat where the $3$-form $\gamma$ is not closed. Moreover if we impose $\gamma$ to be closed, the nearly half-flat structure becomes half-flat. This is the same as setting $\lambda=0$ in \eqref{eq:ng2-intro}.

In \cite{Cleyton-swann} the authors classified the connected Lie groups $G$ that can act as cohomogeneity-one on nearly parallel $\G2$-manifolds. In the case when $G$ is simple, no new solutions were found. One of the possible non-simple Lie groups that can act by cohomogeneity-one on nearly parallel $\G2$-manifolds is $\SU(2)^2$ up to finite quotients. In this case the generic orbit of a complete cohomogeneity-one nearly $\G2$-manifold is isomorphic to $S^3\times S^3$ or its finite quotient. In this article we describe the $\SU(2)^2$-invariant nearly half-flat $\SU(3)$-structures on $S^3\times S^3$. In \cite{Madsen-Salamon} the authors describe the left-invariant half-flat $\SU(3)$-structures on $S^3\times S^3$ using the representation theory of $\SO(4)$ and matrix algebra. In the present article we follow a similar approach and show that we can describe any invariant nearly half-flat $\SU(3)$-structure in terms of two real $3\times 3$ matrices and two real constants satisfying some normalization and commutativity relations (see Theorem \ref{thm:matrixdescr}).

In \cite{Haskins-Lorenzo} the authors constructed the first complete examples of cohomogeneity-one nearly K\"ahler structures in six-dimensions. Hopefully the characterization of invariant nearly half-flat $\SU(3)$-structures on $S^3\times S^3$ can also be used to achieve something similar in the nearly parallel $\G2$-structure case.

The study of the hypersurfaces of $\R^7$ with its associated $\G2$ cross product was initiated by Calabi \cite{Calabi-hypersurface} , Gray \cite{Gray-hypersurface} and later extended to manifolds with $\G2$-structures \cites{Chiossi-Simon,CHIOSSI-Swann, conti-salamon,Madsen-Salamon, spinorial-agricola,  DWIVEDI2019253,Einstein-warpedG2}. The Weingarten map of an oriented hypersurface inside a manifold with a $\G2$-structure can be completely described in terms of the intrinsic torsion forms of the hypersurface. This relation was used in \cite{Cabrera-hypersurface} and \cite{spinorial-agricola} to describe the $\SU(3)$-structure on an oriented hypersurface of a manifold with a $\G2$-structure.

In \textsection\ref{section:setup} we give a brief introduction of $\SU(3)$-, $\G2$-structures on manifolds of dimension $6$, $7$ respectively. We define the respective intrinsic torsion forms and alienate the classes of our particular interests.   

In \textsection\ref{sec:evolution-eqns} we derive the evolution equations that describe the nearly $\G2$-structure on $M\times I$ evolving from a family of nearly half-flat $\SU(3)$-structures on $M$. We then use the evolution equations to describe the nearly half-flat condition in terms of the intrinsic torsion forms. 

In \textsection \ref{sec:setup_half_flat} we specialise the theory curated in the previous sections to $S^3\times S^3$ which constitutes the heart of this article. Using matrix algebra we describe the space of invariant  nearly half-flat $\SU(3)$-structures on $S^3\times S^3$ and show that it can be parameterized using two real $3\times 3$ matrices $P,Q$ and real constants $(a,b)$ satisfying some commutativity and normalization conditions. Further we use this parameterization to describe the moduli space of invariant nearly half-flat $\SU(3)$-structures on $S^3\times S^3$, in terms of which it turns out to be a finite-dimensional symplectic quotient. This
description is rather similar to the half-flat case as described in \cite{Madsen-Salamon}. The description of the invariant nearly half-flat $\SU(3)$-structure $(\omega,\gamma)$ in term of elementary matrices makes it rather elegant to construct $\SU(3)$-structures of specific torsion classes as we do in \textsection\ref{sec:setup_half_flat}. Using this terminology we are able to produce nearly half-flat structures with strictly positive scalar curvature as well as with zero scalar curvature.

In \textsection \ref{sec:evol-eqnsS3} we describe the evolution equations in the special case of $S^3\times S^3$ in terms of the parameterization. The equations we have are in contrast with the half-flat case as presented in \cite{Madsen-Salamon} as they are no longer Painlev\'e equations. In fact the constants $a$ and $b$ evolve with $t$ as opposed to the half-flat case (see \eqref{eqns:evolutionPQ}). We also represent some of the known examples of nearly parallel $\G2$-structures obtained from the invariant nearly half-flat $\SU(3)$-structure on $S^3\times S^3$ such as the homogeneous nearly parallel $\G2$-structure on the Berger space $\SO(5)/SO(3)$ and the sine cone metric on $S^1\times S^3\times S^3$. The two solutions we present here has extra $Z_2^2$ and $\SU(2)$-symmetry respectively. The algebraic setup introduced makes the description of these known nearly parallel $\G2$-structures much more elegant and efficient to use.

\vspace{0.7cm} 
\noindent
\textbf{Acknowledgements.} The author would like to thank Simon Salamon and Thomas Madsen for innumerable suggestions on the project.  The author would also like to thank Benoit Charbonneau for improving the manuscript, and Lorenzo Foscolo, Spiro Karigiannis and Shubham Dwivedi for helpful discussions. Thanks to the anonymous reviewer for many valuable suggestions. A special mention to Simons Collaboration on Special holonomy in geometry, analysis and physics as most of the work on this project was undertaken while the author was a Simons Collaboration postdoc at King's College London. 

\medskip

\section{$\SU(3)$ and $\G2$ structures}\label{section:setup}
We are interested in studying invariant $\SU(3)$-structures on $S^3\times S^3$. An $\SU(3)$-structure on a $6$-dimensional manifold $M$ is defined by a pair $\omega,\Omega$ where $\omega$ is a symplectic form and $\Omega$ is a complex $(3,0)$ form which satisfies the normalization condition 
\begin{align*}
    \Omega\wedge\bar{\Omega}&=-\frac{4i}{3}\omega^3.
\end{align*}
An $\rm{SU}(3)$ reduction defines a circle of real 3-forms $\Gamma:=\{\cos\theta  \Re\Omega+ \sin\theta \Im\Omega, \theta\in\R\}$, and any $\gamma\in\Gamma$ along with the $2$-form $\omega$ defines an almost complex structure $J$, the metric $g$, and the orientation $\vol_g$. Note that an $\SU(3)$-structure on a $6$-dimensional manifold can be defined by a pair $(\omega,\gamma)$ where $\gamma\in \Gamma$. indeed $\gamma$ can determine $J$ as well $J\gamma$ such that $\gamma+i J\gamma$ is a complex holomorphic volume form of type $(3,0)$.  See \cite{Hitchin-3forms} for more details.

Using $\omega$ one can define the symplectic Hodge star $\star:\Omega^rM\to\Omega^{6-r}M$ via the relation
\begin{align}\label{symp-hodgestar}
    \alpha\wedge\star\beta&=\omega(\alpha,\beta) \ \frac{\omega^3}{6}.
\end{align}
Using the above defined symplectic Hodge star for $p\in M$ we have $P_\gamma\in\rm{End}(T^*_pM)$ defined by 
\begin{align*}
    P_\gamma&\colon \alpha \mapsto -\frac{1}{2}\star(\gamma\wedge\star(\gamma\wedge\alpha).
\end{align*}
Then the endomorphism $ J_\gamma=(\det{P_\gamma})^{-\frac{1}{6}}P_\gamma$ defines an almost complex structure on $M$. We write $J$ instead of $J_\gamma$ when there is no scope for confusion. On any point $p\in M$ an almost complex structure $J_\gamma$ on $T_pM$ can be defined using any $\gamma\in\Gamma$ as described in \cite[Section 8.2]{Hitchin-stableforms}. We define $K_\gamma\in \rm{End}(T_pM)\otimes \Lambda_p^6M\cong \rm{End}(T_pM)$ by
\begin{align}\label{J-Hitchin}
    X\mapsto K(X):= (X\lrcorner \gamma)\wedge \gamma \in\Lambda^5_pM\cong T_pM\otimes \Lambda^6_pM.
\end{align} Then $J_\gamma=6K_\gamma/\omega^3$.

The natural action of the Lie group $\SU(3)$ on the tangent space of a $6$-dimensional manifold induces the following decomposition on the space of differential forms $\Omega^p$ into irreducible $\SU(3)$-representations $\Omega^p_k$ with pointwise dimension $k$ (see \cite{Ric-SU3}):
\begin{align}
\begin{split}
    \Omega^2&=\Omega^2_1\oplus\Omega^2_6\oplus\Omega^2_8,\label{form-decomposition}\\
    \Omega^3&=\Omega^3_{\Re}\oplus\Omega^3_{\Im}\oplus\Omega^3_6\oplus\Omega^3_{12},
\end{split}
\end{align} where each summand can be described in terms of the $\SU(3)$-structure as follows,
\begin{align*}
    \Omega^2_1&=\R \omega, \\
    \Omega^2_6&=\{\star(\alpha\wedge\gamma) \ | \  \alpha\in \Omega^1\}=\{\beta\in\Omega^2 \ | \  J\beta=-\beta\},\\
    \Omega^2_8&=\{\beta\in\Omega^2 \ | \  \beta\wedge\gamma=0, \star\beta=-\beta\wedge\omega\}\\
    &=\{\beta\in\Omega^2 \ | \    J\beta=\beta, \beta\wedge\omega^2=0\},
    \end{align*}
    and
    \begin{align*}
         \Omega^3_{\Re} &=\R\gamma,\\
          \Omega^3_{\Im} &=\R  J\gamma,\\
        \Omega^3_6&=\{\alpha\wedge\omega \ | \  \alpha\in\Omega^1\}=\{\xi\in\Omega^3 \ | \  \star\xi=\xi\},\\
        \Omega^3_{12}&=\{\xi\in\Omega^3 \ | \  \xi\wedge\omega=0,\xi\wedge\gamma=0,\xi\wedge  J\gamma=0\}.
    \end{align*}
The space of $1,6$-forms is irreducible and we can describe the space of $4,5$-forms via the isomorphism described by the Hodge star operator $*$.    
Using the decomposition in \eqref{form-decomposition} and the relations between $\omega,\gamma, J\gamma$ one can compute the derivatives of the forms $\omega, \gamma, J\gamma$. For some $w_1^\pm\in\Omega^0, w_2^\pm\in\Omega^2_8$,$w_3\in\Omega^3_{12}$ and  $w_4,w_5\in\Omega^1$ 
\begin{align}
\begin{split}
    \label{eqns:SU-torsion}
    d\omega&= w_1^+ \gamma+w_1^- J\gamma+w_4\wedge\omega+w_3,\\
    d\gamma&=\frac{2w_1^-}{3}  \omega^2+w_5\wedge\gamma+w_2^+\wedge\omega,\\
    d J\gamma&=-\frac{2w_1^+}{3}  \omega^2+Jw_5\wedge\gamma+w_2^-\wedge\omega.
    \end{split}
\end{align} 
The forms $w_i$ are the intrinsic torsion forms of the $\SU(3)$-structures \cites{spinorial-agricola,Ric-SU3,Chiossi-Simon} and define the torsion $T$ of the $\SU(3)$-structure. The vanishing of $T$ implies $\rm{Hol}(\del^{LC})\subseteq \SU(3)$. 

The torsion $T$ of a $G$-structure is the space $T^*M\otimes \mathfrak{g}^\perp$ which for $G=\SU(3)$ turns out to be a $42$-dimensional space. As irreducible $\SU(3)$-modules the torsion space decomposes as follows
\begin{align*}
    T^*M\otimes\mathfrak{su}(3)^\perp \cong \mathcal{W}_1^{\pm}\oplus\mathcal{W}_2^{\pm}\oplus \mathcal{W}_3\oplus\mathcal{W}_4\oplus\mathcal{W}_5.
\end{align*}
The $\SU(3)$-structure is torsion-free or Calabi--Yau if and only if $T=0$ and is nearly K\"ahler if $T\in\mathcal{W}_1^\pm$. For a more thorough description of these torsion classes see \cite{Chiossi-Simon}, \cite{spinorial-agricola}. In this article we are interested in a special torsion class of $\SU(3)$-structures known as nearly half-flat which arise when the $\G2$-structure on $M\times I$ is a nearly parallel $\G2$-structure. The torsion for a nearly half-flat structure lies is $\mathcal{W}_1\oplus \mathcal{W}_2^-\oplus \mathcal{W}_3$. We describe this torsion class in detail in the following section. 

A $\G2$-structure is defined by the reduction of the structure group of the frame bundle of a $7$-dimensional manifold $N$ to the Lie group $\G2\subset \SO(7)$. The existence of a $\G2$-structure on $N$ is characterized by a positive $3$-form $\g2$ preserved by the action of $\G2$ on $\Omega^3(N)$ \cite{Gray1971}. Such a structure exists if and only if the manifold is orientable and spin, conditions which are respectively equivalent to the vanishing
of the first and second Stiefel--Whitney classes. The $3$-form $\g2$ nonlinearly induces a Riemannian metric $g_{\g2}$ and an orientation $\vol_{\g2}$ on $N$ and hence a Hodge
star operator $*_{\g2}$. We denote the Hodge dual $4$-form $*_{\g2}\g2$ by $\psi$. Pointwise we have $\|\g2\|^2=\|\psi\|^2 = 7$, where the norm is taken with respect to the metric induced by $\g2$.  

Similar to $\SU(3)$-structure, a $\G2$-structure on $N$ induces a splitting of the spaces of differential forms into irreducible $\G2$ representations.  The space of $2$-forms $\Omega^2$ and $3$-forms $\Omega^3$ decompose as 
\begin{align*}
\Omega^2&=\Omega^2_7\oplus \Omega^2_{14}, \\
\Omega^3&=\Omega^3_1\oplus \Omega^3_7\oplus \Omega^3_{27}.
\end{align*}
 More precisely, we have the following description of the space of forms : \begin{align*} 
 \Omega^2_7 &=\{X\lrcorner \g2\mid X\in \Gamma(TN)\} = \{\beta \in \Omega^2\mid *(\g2\wedge \beta)=2\beta\} , \\
 \Omega^2_{14} &=\{\beta \in \Omega^2(N)\mid \beta \wedge \psi =0 \} = \{\beta\in \Omega^2\mid *(\g2\wedge \beta)=-\beta\}. 
 \end{align*}

Similarly, for $3$-forms
\begin{align*}
\Omega^3_1 &=\{ f\g2 \mid f\in C^{\infty}(N)\}, \\
\Omega^3_7 & = \{ X\lrcorner \psi \mid X\in \Gamma(TN)\} = \{*(\alpha \wedge \g2) \mid \alpha \in \Omega^1\}, \\
\Omega^3_{27} & = \{ \eta \in \Omega^3(N) \ \mid \ \eta\wedge \g2 = 0 = \eta\wedge \psi\}.
\end{align*}
The decomposition of $\Omega^4$ and $\Omega^5$ are obtained by taking the respective Hodge stars with respect $*_\g2$.

Given a $\G2$-structure $\g2$ on $M$, we can decompose $d\g2$ and $d\psi$ according to the above decomposition. This defines the \emph{torsion forms}, which are unique differential forms $\tau_0 \in \Omega^0$, $\tau_1 \in \Omega^1$, $\tau_2 \in \Omega^2_{14}$ and $\tau_3 \in \Omega^3_{27}$ such that (see \cite{skflow})
\begin{align*}
d\g2 &= \tau_0\psi + 3\tau_1\wedge \g2 + *_{\g2}\tau_3,  \\
d\psi &= 4\tau_1\wedge \psi + *_{\g2} \tau_2.
\end{align*}
Here the torsion lives in the $49$-dimensional space $T^*N\otimes \mathfrak{g}_2^\perp$ and decomposes into irreducible $\G2$-modules as follows
\begin{align*}
    T^*N\otimes \mathfrak{g}_2^\perp \cong \mathcal{X}_0\oplus \mathcal{X}_1\oplus \mathcal{X}_2\oplus\mathcal{X}_3.
\end{align*}
These torsion forms give rise to the sixteen classes of $\G2$-structures and $T=0$ if and only if $d\g2=d\psi=0$ (see \cite{Fernandez-Gray}, \cite{classification_G2}). The torsion class for which $T\in \mathcal{X}_0$ is called nearly parallel $\G2$-structure. 
\begin{defn}
A $\G2$-structure $\g2$ is {\bf{nearly parallel }}if and only if there exists $\lambda\neq 0$ such that
\begin{align}
d\g2=\lambda\psi \ \ \ \ \textup{and} \ \ \ \ d\psi=0. \label{eq:ng2defn}
\end{align} 
\end{defn}

\noindent
In this case, $T_{ij}=\dfrac{\lambda}{4}(g_\g2)_{ij}$. 

\begin{rem}
If $\g2$ is a nearly $\G2$-structure differentiating \eqref{eq:ng2defn} gives $d\lambda\wedge \psi =0$ which implies $d\lambda =0$, as wedge product with $\psi$ is an isomorphism from $\Omega^1_7$ to $\Omega^5_7$. Thus, if $N$ is connected $\lambda$ is a constant.

In this article we are interested in parameterizing the $\SU(3)$-structures that arise on the equidistant orientable hypersurfaces of manifolds with nearly parallel $\G2$-structures.
\end{rem}
\section{Evolution equations from $\SU(3)$ to $\G2$}\label{sec:evolution-eqns}
The exceptional Lie group $\G2$ is the group of automorphisms of the Octonions $\mathbb{O}$ that preserves the splitting $\mathbb{O}\cong \R+\rm{Im}\mathbb O$. Then one can define the Lie group $\SU(3)$ as the subgroup of $\G2$ that preserves an imaginary unit Octonion. This fact indicates the presence of an $\SU(3)$-structure on an orientable hypersurface of a manifold with $\G2$-structure which led Calabi \cite{Calabi-hypersurface} and Gray \cite{Gray-hypersurface} study the induced $\SU(3)$-structure on orientable hypersurfaces of $\rm{Im} \mathbb{O}$.   

 Let $I\subset\R$ be an interval. Given a one-parameter family of $\SU(3)$-structures $(\omega(t),\gamma(t))$ on $M$, one can define a $\G2$-structure $(\g2,\psi)$ on $I\times M$ by 
 \begin{align}
     \begin{split}\label{g2-eqns}
         \g2&=dt\wedge\omega(t)+\gamma(t)\\
         \psi=*_\g2\g2&=\frac{1}{2}\omega^2(t)-dt\wedge J\gamma(t).
     \end{split}
 \end{align}
 
 \noindent
\textbf{Conditions on $\omega,\gamma$ for nearly $\G2$-structure.} Now suppose $\g2,\psi$ define a nearly $\G2$-structure, that is for some non-zero constant $\lambda\in\R$ we  have   
\begin{align*}
    d\g2&=\lambda\psi.
\end{align*} From \eqref{g2-eqns} we get that
\begin{align}
    \label{dphi}d\g2&=dt\wedge(- d\omega(t)+\gamma'(t))+d\gamma(t),\\
   \label{dpsi} d\psi&=\frac{ d(\omega^2(t))}{2}+dt\wedge(\frac{1}{2}(\omega^2)'(t)+dJ\gamma(t)).
\end{align} Thus $(\g2,\psi)$ defines a nearly $\G2$-structure if and only if 
\begin{align}
\begin{split}\label{SU3FORNG2}
       d\omega(t)&=\gamma'(t)+{\lambda} J\gamma(t),\\
        d\gamma(t)&=\frac{\lambda}{2}\omega^2(t)\\
        dJ\gamma(t)&=-\frac{(\omega^2)'}{2}.  
\end{split}
    \end{align}
  For some $\alpha,\beta \in C^\infty(M), Z\in \Omega^1(M), \gamma'_{12}\in \Omega^3_{12}$ the $3$-form $\gamma'=\alpha\gamma+\beta J\gamma+Z\wedge\omega+\gamma'_{12}$.
  
    Equation \eqref{SU3FORNG2} immediately imply $d\omega^2=0$ which further implies that $\omega\wedge d\omega=0$. Thus we get
    \begin{align*}
      0&= \omega\wedge d\omega = \omega\wedge\gamma'+\lambda\omega\wedge J\gamma.
    \end{align*} Since $\omega\wedge J\gamma = 0$ the above implies that $\omega\wedge\gamma'=0$ hence $\beta=0$. Since $\gamma\wedge\omega=0$ we have
    \begin{align*}
        d\gamma\wedge\omega&=\gamma\wedge d\omega,
    \end{align*} which from \eqref{SU3FORNG2} implies
    \begin{align*}
        \frac{\lambda}{2}\omega^3&=\gamma\wedge\gamma'+\lambda (\gamma\wedge J\gamma)=\gamma\wedge\gamma'+\frac{2}{3}\lambda\omega^3.
    \end{align*} 
    
Thus, $\gamma\wedge\gamma'=(-1/6)\lambda\omega^3$ which implies $\alpha=-\frac{\lambda}{4}$ and we obtain 
\begin{align*}
    \gamma'&=\alpha\gamma-\frac{\lambda J\gamma}{4}+\gamma'_{12}.
\end{align*}

Now let $\omega'=p\omega+X\lrcorner\gamma+\omega'_{8}$ for some $p\in C^{\infty}(M), X\in \Gamma(TM),\omega'_8\in \Omega^2_8$. 
The equation $d\psi=0$ implies $dJ\gamma=-\omega'\wedge\omega$ which further implies $p=2\alpha/3$ and $X=0$. 
Substituting $\alpha=w_1^+, \gamma'_{12}=w_3,$ and $\omega'_8={w}_2^-$ in \eqref{SU3FORNG2} we get that 
\begin{align}
    \begin{split}
        \label{su3-torsionforNG2} d\omega&=w_1^+\gamma+\frac{3\lambda}{4}J\gamma+w_3,\\
        d\gamma&=\frac{\lambda}{2}\omega^2,\\
        dJ\gamma&=-\frac{2}{3}w_1^+\omega^2+w_2^-\wedge\omega.
         \end{split} 
\end{align}
Thus it is clear that for nearly half-flat $\SU(3)$-structures the only non vanishing torsion forms are $w_1^\pm$, $w_2^-$ and $w_3$. But since $w_1^-$ is a constant completely determined by the nearly parallel $\G2$-structure, the dimension of the unknown torsion is given by $\dim(\R)+\dim(\Omega^2_8)+\dim(\Omega^3_{12})=21$, which is rather similar to "\textit{half-flat}" $\SU(3)$-structures (\cite{Madsen-Salamon}) in the sense that $W_1^-$ is the only extra non-vanishing torsion. Also observe that none of the forms $\omega,\gamma,J\gamma$ are necessarily closed as opposed to the half-flat case where $\gamma$ is always closed.  We call the $\SU(3)$-structures whose torsion is given by \eqref{su3-torsionforNG2}, "\textit{nearly half-flat}" $\SU(3)$-structures.

The condition $d\gamma=\frac{\lambda}{2}\omega^2$ implies $d\omega^2=0$ and forces the torsion terms $w_2^+, w_4$ and $w_5$ to vanish  in \eqref{eqns:SU-torsion}. Thus $d\gamma=\frac{\lambda}{2}\omega^2$ is a sufficient condition for an $\SU(3)$-structure on a $6$-dimensional manifold to be nearly half-flat as first described in \cite{nearlyhypo-nhf}.
\begin{defn}
 
    An $\SU(3)$-structure $(\omega,\gamma)$ on a 6-manifold $N^6$ is nearly half-flat if for some non-zero real constant $\lambda$
    \begin{align*}
        d\gamma&= \frac{\lambda}{2}\omega^2.
    \end{align*}
   
\end{defn}

 \begin{table}[h!]\label{Table:torsionforms}
     \centering
     \begin{tabular}{|w{c}{2cm}|w{c}{2cm}|}
     \hline
          \cellcolor{blue!20}${\mathcal{W}_1^+}$&\cellcolor{red!20}${\mathcal{W}_1^-}$  \\
         \hline
          ${\mathcal{W}}_2^+$ &\cellcolor{blue!20} $\mathcal{W}_2^- $ \\
          \hline
         \multicolumn{2}{|c|}{\cellcolor{blue!20}${\mathcal{W}_3}$}  \\
          \hline
          \multicolumn{2}{|c|}{${\mathcal{W}_4}$}\\
          \hline
           \multicolumn{2}{|c|}{$\mathcal{W}_5$}\\
           \hline
     \end{tabular}
     \caption{The shaded cells of the table represents the non-zero torsion classes for a nearly half-flat structure. The different shade for the torsion class $\mathcal{W}_1^-$ emphasises that it is a non-zero constant. }
 \end{table}
Now we can state the result originally proved in \cite[Proposition 5.2]{nearlyhypo-nhf}

\begin{prop}\label{prop:su3tog2}
   Let $I\subseteq \R$ parameterised by $t$. A family of nearly half-flat structures $(\omega(t),\gamma(t))$ on $N^6$ can be lifted to a nearly $\G2$-structure $\g2=dt\wedge\omega+\gamma$ on $N^6\times I$ if and only if $(\omega,\gamma)$ satisfy the evolution equation
   \begin{align}
   \begin{split} \label{eqns:evolution}
       \gamma'(t)&= d\omega(t)-\lambda J\gamma(t).
   \end{split}
   \end{align}
\end{prop}
One should note that if $(\omega(t),\gamma(t))$ is nearly half-flat for all $t\in I$ then one need not impose the condition $dJ\gamma=-\omega'\wedge\omega$ as it can be obtained from differentiating the equation $d\gamma(t)=\lambda/2 \omega^2(t)$ with respect to $t$ and use the evolution equation mentioned in the above proposition. If only the initial $\SU(3)$-structure is nearly half-flat then the equation $dJ\gamma=-\omega'\wedge\omega$ implies that $(\omega(t),\gamma(t)$ is nearly half-flat for all time $t\in I$.

In \cites{cvlt-liftnhf, Conti-embedding} established the existence of a solution of
the system \eqref{eqns:evolution} for all time given the initial nearly half-flat structure is real analytic. This result is analogous to the result of Bryant \cite{Bryant-non-embedding} in the half-flat case where he proves that unless one assumes real analytic initial data, Hitchin's flow equations do not necessarily admit a solution. For more on the existence of the lifts of the nearly half-flat structure to a nearly parallel $\G2$-structure see \cite{liftingsu3tonhf, Fabian-thesis}.


\begin{rem}
If we put $\lambda=0$ (i.e. torsion-free) in \eqref{su3-torsionforNG2} we get back the half-flat conditions as we should! 
\end{rem}    

In the above notation for intrinsic torsion forms the scalar curvature $s$ of the Levi-Civita connection for nearly half-flat $\SU(3)$-structures is given by 
\begin{align}\label{eqn:scalar-nhf}
    s&= \frac{10}{3}(w_1^+)^2+\frac{15\lambda^2}{8}-\frac{1}{2}|w_2^-|^2-\frac{1}{2}|w_3|^2.
\end{align} The general expression for the scalar curvature in terms of the intrinsic torsion of the $\SU(3)$-structure was derived in \cite[Theorem 3.4]{Ric-SU3}.
\begin{lemma}\label{lemma:w3_w2}
    On a connected manifold $N^6$ with a nearly half-flat structure $(\omega,\gamma)$
    \begin{enumerate}
      \item $dw_3=0$ implies that $w_2^-=0$ and $w_1^+$ is constant.
        \item $w_2^-=0$ implies that $w_1^+$ is constant.
       \end{enumerate}
\end{lemma}
\begin{proof}
    Both the assertions follow from \eqref{su3-torsionforNG2}. The first follows from equating $d(d\omega)=0$ which implies
    \begin{align*}
        dw_3&= -dw_1^+\wedge\gamma-w_2^-\wedge\omega.
    \end{align*}
    The second follows by using $d\omega^2=0$ in $d(dJ\gamma)=0$. 
\end{proof}

Thus if $w_3=0$ the torsion of the nearly half-flat structure is in $\mathcal{W}_1$ and is constant. Moreover for a nearly half-flat structure $\mathcal{W}_1^-$ is always non-zero so the only possible torsion classes for a nearly half-flat structure are 

\begin{align*}
    \mathcal{W}_1^-, \  \mathcal{W}_1,\ \mathcal{W}_1^-+\mathcal{W}_3 , \ \mathcal{W}_1+\mathcal{W}_3,\ \mathcal{W}_1^-+\mathcal{W}_2^-+\mathcal{W}_3, \ \mathcal{W}_1+\mathcal{W}_2^-+\mathcal{W}_3.    
\end{align*} 

Also note that of the torsion class for $(\omega,\gamma)$ is in $\mathcal{W}_1$ then there exists a $\gamma_\theta\in \{\cos(\theta)\gamma+\sin{(\theta)}J\gamma \ | \ \theta\in \R\}$ such that $(\omega,\gamma_\theta)$ is nearly K\"ahler.

An $\SU(3)$-structure $(\omega,\gamma)$ is half-flat if and only if the forms $\gamma$ and $\omega^2$ are closed. Under some special hypothesis a nearly half-flat structure $(\omega,\gamma)$ can be rotated to a half-flat structure.

\begin{prop}\label{prop:nhf-hf}
Let $(\omega,\gamma)$ be a nearly half-flat $\SU(3)$-structure such that $w_2^-=0$. Then for $\theta=\tan^{-1}\left(\frac{3\lambda}{4w_1^+}\right)$ the $\SU(3)$-structure $(\omega,\gamma_\theta:=\cos(\theta)\gamma+\sin(\theta)J\gamma)$ is half-flat.
\end{prop}

\begin{proof}
Since $d(\omega^2)=0$, the $\SU(3)$-structure $(\omega,\gamma_\theta)$ is half-flat if and only if $d\gamma_\theta=0$. 

From Lemma \ref{lemma:w3_w2} we know that  $w_2^-=0$ implies $dw_1^+=0$ hence
\begin{align*}
    d\gamma_\theta&= \cos(\theta) d\gamma+\sin(\theta) d J\gamma\\
    &= \frac{\cos(\theta) \lambda}{2}\omega^2-\frac{2\sin(\theta)}{3} w_1^+\omega^2.
\end{align*}
This implies that $d\gamma_\theta=0 \iff \tan(\theta)=\frac{3\lambda}{4w_1^+}$. 
\end{proof}

Thus given a nearly half-flat $\SU(3)$-structure $(\omega,\gamma)$ satisfying the above hypothesis there is a half-flat $\SU(3)$-structure that induces the same almost complex structure and metric as $(\omega,\gamma)$.

If we further assume $w_1^+=0$ in the above proposition we get that the $\SU(3)$-structure $(\omega,J\gamma)$ is half-flat which we already know from \eqref{su3-torsionforNG2}. 

\begin{rem}
   The angle described in Proposition \ref{prop:nhf-hf} can also be related to the Special Lagrangian phase angle. Since the $3$-form $\gamma_\theta$ as defined in Proposition \ref{prop:nhf-hf} is closed one can define special Lagrangian submanifolds on $(N^6,\omega,J)$ calibrated by $\gamma_\theta$.  
\end{rem}




\section{Parameterising invariant nearly half-flat structures on $S^3\times S^3$}\label{sec:setup_half_flat}
Let $M$ denote the six dimensional manifold $S^3\times S^3$. We will use the notation as used in \cite{Madsen-Salamon}. The tangent bundle $TM$ is trivial since $M$ is a Lie group. Thus $TM\cong M\times \R^6\cong M\times \so({4})\cong M\times \su(2)\oplus\su(2)$. We will denote by $A,B$ the $2$ copies of $\su(2)$ in the cotangent bundle at the identity, $T_{\id}^*M\cong A\oplus B$. We choose bases $\{e^1,e^3,e^5\}$ and $\{e^2,e^4,e^6\}$ for $A$ and $B$ respectively such that \begin{align*}
    &de^1=e^{35}, \ de^3=-e^{15}, \ de^5=e^{13} \\
     &de^2=e^{46}, \ de^4=-e^{26}, \  de^6=e^{24}.
\end{align*}

We now describe the invariant nearly half-flat structures on $S^3\times S^3$ in terms of $3\times 3$ real matrices similar to \textsection 3 in \cite{Madsen-Salamon}. Using the same notation for $T^*M=A\otimes B =\colon U \cong M_{3\times 3}(\R)$ where $A,B\cong \su(2)$ we have 
\begin{align*}
    \Omega^2M&\cong \Omega^2A\oplus (A\otimes B)\oplus \Omega^2B,\\
    \Omega^3M &\cong \Omega^3A\oplus(\Omega^2A\otimes B)\oplus (A\otimes\Omega^2B)\oplus\Omega^3B. 
\end{align*} Since $d(\omega^2)=0$ from \eqref{su3-torsionforNG2} we have $\omega^2\in\Omega^2A\otimes\Omega^2B$. Since $\omega^2$ is also non-degenerate, $\omega\in A\otimes B$ . Thus $\omega$ can be represented by a $3\times 3$ real matrix $P$ by $\omega=\sum_{i,j=1}^3P_{ij}e^{2i-1}\wedge e^{2j}$. The non-degeneracy of $\omega$ implies that $\det P \neq 0$ anywhere. Moreover, the $4$-form $\omega^2=-2\sum_{i,j=1}^3 {\mathrm{Ad}}(P^T)_{ij}de^{2i-1}\wedge de^{2j}$.

We define a $3$-form $\delta\in (\Omega^2A\otimes B)\oplus (A\otimes\Omega^2B)$ such that $\delta\wedge\omega=0$ and $d\delta=\omega^2$. A symmetric choice of $\delta$ in $A,B$ yields
\begin{align*}
    \delta&= -\sum_{i,j=1}^3\rm{Adj}(P^T)_{ij}(de^{2i-1}\wedge e^{2j}+e^{2i-1}\wedge de^{2j}).
\end{align*}
Since $d\gamma=\lambda/2\omega^2$ for nearly half-flat structures the $3$-form $\gamma-\frac{\lambda}{2}\delta$ is closed and hence for some $a,b\in \R$ and $d\beta\in(\Omega^2A\otimes B) \oplus (A\otimes\Omega^2B)$ we can assume $\gamma$ to be of the form 
\begin{align*}
    \gamma&= ae^{135}+be^{246}+d\beta+\frac{\lambda}{2}\delta.
\end{align*}
The $2$-form $\beta\in A\otimes B$ can be represented by a $3\times 3$ real matrix $Q$ and we have $\beta=\sum_{i,j=1}^3Q_{ij}e^{2i-1}\wedge e^{2j}$. The $\Omega^2A\otimes B, A\otimes\Omega^2B$ components of $\gamma$ are given by matrices $Q_1,Q_2$ respectively where 
\begin{align}
\begin{split}\label{eqns:Q1Q2}
    Q_1&=Q-\frac{\lambda}{2}\rm{Adj}(P^T),\\
    Q_2&=-Q-\frac{\lambda}{2}\rm{Adj}(P^T),
\end{split}
\end{align} and we have 
\begin{align*}
    \gamma&=ae^{135}+be^{246}+\sum_{i,j=1}^3\left((Q_1)_{ij}de^{2i-1}\wedge e^{2j}+(Q_2)_{ij}e^{2i-1}\wedge de^{2j}\right).
\end{align*}
The identity $\gamma\wedge\omega=0$ implies that $Q^TP$ is symmetric which also implies $Q_i^TP$ is symmetric for $i=1,2$.
\medskip

By using \eqref{J-Hitchin} we compute the almost complex structure $J_\gamma$ in terms of $a,b,Q_1,Q_2$. For $r=1,2,3$,
\begin{align}
\begin{split}\label{eqn:J-defn}
    Je^{2r-1}&= \frac{1}{\det(P)}\left((ab-\tr({Q_2Q_1^T}))e^{2r-1}+2\sum_{i=1}^3\left((Q_2Q_1^T)_{ri}e^{2i-1}-(aQ_2-{\rm{Adj}}(Q_1^T))_{ri}e^{2i}\right)\right),\\
    Je^{2r}&= \frac{1}{\det(P)}\left(-(ab-\tr({Q_1^TQ_2}))e^{2r}+2\sum_{i=1}^3\left((bQ_1^T-{\rm{Adj}}(Q_2))_{ri}e^{2i-1}-(Q_1^TQ_2)_{ri}e^{2i}\right)\right).    
\end{split}
    \end{align}
    
For $i=1,\dots,6$,
\begin{align*}
    J^2 e^i&=\frac{1}{(\det P)^2}((ab-\tr(Q_1^TQ_2))^2+4(a\det Q_2+b\det Q_1)-4\tr{(\rm{Adj}(Q_1^TQ_2))})e^i.
\end{align*}

Since $J^2=-\id$ we have 

\begin{align}\label{eqn:detP}
    \det P = (-(ab-\tr(Q_1^TQ_2))^2-4(a\det Q_2+b\det Q_1)+4\tr{(\rm{Adj}(Q_1^TQ_2))})^\frac{1}{2},
\end{align} which denotes the normalization condition for the nearly half-flat $\SU(3)$-structure.

\medskip

 The space of invariant nearly half-flat structures can now be parameterized by $(\lambda,a,b,P,Q)$ satisfying the commutativity relation $P^TQ=Q^TP$ and the normalization condition \eqref{eqn:detP}. 
 
 We denote the set of invariant nearly half-flat structures corresponding to a fixed $\lambda\in \R^*$ and a cohomology class $(a,b)$ for $\gamma-\frac{\lambda}{2}\delta$ by $\mathcal{H}_{\lambda,a,b}$. Using the isomorphism between  $M_{3\times 3}(\R)$ the space of real $3\times 3$ matrices and the space of real symmetric trace-free $4\times 4$ matrices $V$ we can describe the space $\mathcal{H}_{\lambda,a,b}$ as the kernel of the $\rm{SO}(4)$-equivariant map described in \cite[Section 3]{Madsen-Salamon}. Under the isomorphism between $M_{3\times 3}(\R)$ and $V$ the condition $Q^TP-P^TQ=0$ can be written as $[Q,P]=0$. Using this formulation we can describe the space of invariant nearly half-flat structures $\mathcal{H}_{\lambda,a,b}$ as a kernel of the moment map defined below. The proof follows verbatim as described in \cite[Theorem 1]{Madsen-Salamon} and hence omitted. 

\begin{thm}\label{thm:matrixdescr}
     The space of invariant nearly half-flat structures $\mathcal{H}_{\lambda,a,b}$ can be described as
     \begin{align*}
         \{(Q,P)\in U\oplus U \ | \ P,Q \ \text{satisfy} \ \eqref{eqn:detP}, \text{and} \  Q^TP =P^TQ  \}.
     \end{align*}
\end{thm}

With respect to the natural symplectic structure on $T^*V$ the map 
\begin{align*}
    \mu \colon V\oplus V &\to \mathfrak{so}(4) \cong \Omega^2\R^4 \\
    (A,B)& \mapsto [A,B]
\end{align*} defines a moment map for the Hamiltonian action of $\rm{SO}(4)$ on $T^*V$. The set $H_{\lambda,a,b}$ then lies in the kernel of $\mu$.

\begin{corr}\label{corr:modulispace}
    Modulo equivalent relations $\mathcal{H}_{\lambda,a,b}$ is a subset of the singular symplectic quotient 
    \begin{align*}
        \frac{\mu^{-1}(0)}{\rm{SO}(4)}\cong \frac{\R^3\oplus\R^3}{S_3}.
    \end{align*}
\end{corr}

For describing the flow equations \eqref{SU3FORNG2} in this matrix framework we need to compute $J\gamma \in \mathcal{H}_{\lambda,a,b}$. To do this we make use of the following identity
\begin{align*}
    J\gamma(X,Y,Z)&=\gamma(JX,JY,JZ) = -\gamma(JX,Y,Z),
\end{align*} and can compute 
\begin{align*}
    J\gamma=&\frac{2}{\det{P}}\Bigl((a\tr(Q_1^TQ_2)-2\det Q_1-a^2b)e^{135}-(b\tr(Q_1^TQ_2)-2\det Q_2-ab^2)e^{246} \\
    &-\sum_{i,j=1}^3((ab+\tr{(Q_1^TQ_2)}) \ Q_1-2a \ {\textup{Adj}}(Q_2^T)-2Q_1Q_2^TQ_1)_{ij}de^{2i-1}\wedge e^{2j}\\
    &+\sum_{i,j=1}^3((ab+\tr{(Q_1^TQ_2)}) \ Q_2-2b \ {\textup{Adj}}(Q_1^T)-2Q_2Q_1^TQ_2)_{ij}e^{2i-1}\wedge de^{2j}\Bigl).
\end{align*}
We denote by \begin{align}
\begin{split}\label{eqns:ABR1R2}
A&\coloneqq a\tr(Q_1^TQ_2)-2\det Q_1-a^2b,\\
B&\coloneqq -(b\tr(Q_1^TQ_2)-2\det Q_2-ab^2),\\
    R_1&\coloneqq -((ab+\tr{(Q_1^TQ_2)}) \ Q_1-2a \ {\textup{Adj}}(Q_2^T)-2Q_1Q_2^TQ_1), \\
    R_2&\coloneqq (ab+\tr{(Q_1^TQ_2)}) \ Q_2-2b \ {\textup{Adj}}(Q_1^T)-2Q_2Q_1^TQ_2,      
\end{split}
\end{align} thus we can write 
\begin{align*}
    J\gamma&= \frac{2}{\det P}\Bigl(A e^{135}+B e^{246} +\sum_{i,j=1}^3\left((R_1)_{ij}de^{2i-1}\wedge e^{2j}
    +(R_2)_{ij}e^{2i-1}\wedge de^{2j}\right)\Bigl).
\end{align*}
One can also check that $\gamma\wedge J\gamma=2/3 \omega^3$ using \eqref{eqn:detP} and $J\gamma\wedge\omega=0$ which uses the fact that $P^TQ_i$ is symmetric for $i=1,2$. 

The $4$-form 
\begin{align*}
    d(J\gamma)&=\frac{2}{\det P}\sum_{i,j=1}^3 R_{ij}de^{2i-1}\wedge de^{2j},
\end{align*}
where \begin{align*}
    R= R_1+R_2=(ab+\tr{(Q_1^TQ_2)})(Q_2-Q_1)-2(b {\textup{Adj}}(Q_2^T)-a{\textup{Adj}}(Q_1^T))-2Q_2Q_1^TQ_2+2Q_1Q_2^TQ_1.
\end{align*}
\begin{rem}
    In \cite{Madsen-Salamon} the authors did similar computations when the $\SU(3)$-structure is half-flat or equivalently when $\lambda=0$ which implies $Q_1=-Q_2=Q$. Our results here matches that of in \cite{Madsen-Salamon} for $\lambda=0$. Also note that they used an isomorphism between the space of real $3\times 3$ matrices and the space of real symmetric trace-free $4\times 4$ matrices to simplify their expressions but in this case the isomorphism does not simplifies the computations or the expressions significantly and is thus not used in computations. 
    \end{rem}

The $6$-form $d\omega \wedge J\gamma=-dJ\gamma\wedge \omega =  \frac{2}{\det P}\tr(P^TR)\vol_6$. From \eqref{su3-torsionforNG2} this implies \begin{align*}
    w_1^+&= \frac{\tr(P^TR)}{2 (\det P)^2}.
\end{align*}
Thus we can rewrite \eqref{su3-torsionforNG2} as 
\begin{align*}
    \frac{\det P}{2} P_{ij}(de^{2i-1}\wedge e^{2j}-e^{2i-1}\wedge de^{2j}) =&\left(\frac{\tr(P^TR)}{4\det P}a+\frac{3\lambda}{4} A\right)e^{135}+\left(\frac{\tr(P^TR)}{4\det P}b+\frac{3\lambda}{4} B\right)e^{246}\\
   &+\left(\frac{\tr(P^TR)}{4\det P}Q_1+\frac{3\lambda}{4} R_1\right)_{ij}de^{2i-1}\wedge e^{2j}\\
   &+\left(\frac{\tr(P^TR)}{4\det P}Q_2+\frac{3\lambda}{4} R_2\right)_{ij}e^{2i-1}\wedge de^{2j} +\frac{\det P}{2} w_3,\\
   R_{ij}de^{2i-1}\wedge de^{2j}&= \frac{\tr(P^TR)}{3\det P}{\rm{Adj}}(P^T)_{ij}de^{2i-1}\wedge de^{2j}+w_2^-\wedge\omega,
\end{align*}
and make the following observations
\begin{prop} Let $(\omega,\gamma)\in \mathcal{H}_{\lambda,a,b}$. 
\begin{enumerate}[i)]
    \item If $w_1^+=w_3=w_2^-=0$, the $\SU(3)$-structure satisfies
\begin{align*}
    d\omega&= \frac{3\lambda}{4}J\gamma,\ \ \ 
    d\gamma= \frac{\lambda}{2}\omega^2. 
\end{align*}
Thus the nearly half-flat structure is nearly K\"ahler if  \begin{align*}
A &= B=0,\\
    R_1 &=-R_2= \frac{2\det P}{3\lambda}P.
\end{align*}
Note that $R=R_1+R_2$ so $R=0$ in this case.
\item The torsion form $w_1^+=0$ if and only if $\tr(P^TR)=0$. 

\item The torsion form $w_2^-=0$ that is the $\SU(3)$-structure is \textit{"nearly" co-coupled }if and only if
\begin{align*}
    R&=\frac{\tr(P^TR)}{3\det P}{\rm{Adj}}(P^T)
\end{align*}  
\item The torsion  form $w_3=0$ or the $\SU(3)$-structure is \textit{"nearly" coupled }if and only if
\begin{align*}
    A&=-\frac{\tr(P^TR)}{3\lambda\det P}a, \ \ \ B=-\frac{\tr(P^TR)}{3\lambda \det P}b,\\
    R_1&=\frac{1}{3\lambda}\left(2\det P P -\frac{\tr(P^TR)}{\det P}\ Q_1\right),\ \ \ R_2=-\frac{1}{3\lambda}\left(2\det P P +\frac{\tr(P^TR)}{\det P}\ Q_2\right).
\end{align*}
\end{enumerate}   
\end{prop}
\medskip

We can now use the above algebraic framework to describe some examples of nearly half-flat $\SU(3)$-structures on $S^3\times S^3$.

\medskip

\noindent
\textbf{\textit{Nearly K\"ahler solution}}: If we assume $P=\rm{diag}(p_1,p_2,p_3)$, and $Q=\rm{diag}(q_1,q_2,q_3)$ the nearly half-flat structure given by $(a,b,P,Q)$ solves the nearly K\"ahler equations for a constant  $\lambda$ only when 
\begin{align}\label{eqn:nk-soln}
    (P,Q)&=\left(\pm\frac{4\sqrt{3}}{9\lambda^2} \rm{Id},0\right), \ \ \ a=b=\frac{16}{27\lambda^3}.
\end{align}

In the current framework the above solution represents the \textit{unique} $S^3\times S^3$-invariant  nearly K\"ahler solution ( compare with \cite[Proposition 3]{Madsen-Salamon})

\medskip

\noindent
\textbf{\textit{Examples of type $\mathcal{W}_1^++\mathcal{W}_1^-$}}:   
  The nearly half-flat structure has torsion contained in  $\mathcal{W}_1$ if and only if $R=-\frac{2}{3}w_1^+ {\rm{Adj}}(P^T)$. Furthermore if $w_1^+=0$ the structure becomes nearly K\"ahler which we have already seen above so we assume $w_1^+\neq 0$. Assuming $P=p{\rm{Id}},Q = q{\rm{Id}}$ we obtain the following solutions 
 for $|p|\in \left(0,\frac{4\sqrt{3}}{9\lambda^2}\right)$
\begin{align*}
    a=b=\lambda p^2, \ \ \ q=\pm \frac{p\sqrt{12\sqrt{3}|p|-27\lambda^2 p^2}}{6}.
\end{align*}

For the above solutions $w_1^+=\frac{\sqrt{3}q }{p^2}$. From \eqref{eqn:scalar-nhf} one can see that examples of this type has strictly positive scalar curvature, which for the above solution is given by 
\begin{align*}
    s&=\frac{10q^2}{p^4}+\frac{15\lambda^2}{8} =\frac{10\sqrt{3}}{3|p|}-\frac{45\lambda^2}{8}.
\end{align*}
\medskip
\noindent

\noindent
\textbf{\textit{Examples of type $\mathcal{W}_1^-+\mathcal{W}_3$}}: The nearly half-flat structure in this torsion form satisfy $dJ\gamma =0$.  For $\lambda=4$, $P=p{\rm{Id}},Q = q{\rm{Id}}$ and $a>\frac{1}{256}$ the following nearly half flat structure has $w_1^+=w_2^-=0$ 
\begin{align*}
    b=\frac{512a^2}{256a-1}, \ \ \  q=\frac{128a^2}{256a-1}, \ \ \ p=\pm 8a \sqrt{\frac{1}{256a-1}}
\end{align*}

\begin{rem}
    The above example with $w_1^+=w_2^-=0$ can be used to construct a $\G2 T$-structure in the sense of \cite[Proposition 4.2]{fino2024twisted} on $S^3\times S^3\times \R$.
\end{rem}
\medskip

\noindent
\textbf{\textit{Zero scalar curvature metric}} If we assume $b=a=0, P=p\rm{Id}$ and $Q=q\rm{Id}$ then the normalization condition \eqref{eqn:detP} becomes \begin{align*}
    3q^4-24q^2p^4+48p^8-p^6=0,
\end{align*}
which has the following solutions

\begin{align*}
    q=\pm\frac{p\sqrt{36p^2\pm 3\sqrt{3}p}}{3}.
\end{align*}
The nearly half-flat metric in this case is given by 
\begin{align*}
    g&=\frac{2(2p^2-q)^2}{p^2}e^{2i-1}\otimes e^{2i-1}+\frac{2(2p^2+q)^2}{p^2}e^{2i}\otimes e^{2i}+\frac{4p^4-q^2}{p^2}(e^{2i-1}\otimes e^{2i}+e^{2i}\otimes e^{2i-1}).
\end{align*}
One can compute the scalar curvature of the nearly half-flat structure metric by \eqref{eqn:scalar-nhf}. For $a=b=0$ and $q=\pm\frac{p\sqrt{36p^2+ 3\sqrt{3}p}}{3}$ the scalar curvature takes the form
\begin{align*}
    s&=\frac{2(72p^4+105p+5\sqrt{3})}{3p}.
\end{align*} There are two values of $p$ for which $s=0$ but for only one of them $q=\pm\frac{p\sqrt{36p^2+ 3\sqrt{3}p}}{3}\in\R$ hence we get one solution from this case. 

However for $q=\pm\frac{p\sqrt{36p^2-3\sqrt{3}p}}{3}$ the scalar curvature turns out to be \begin{align*}
    s&=\frac{2(72p^4+105p-5\sqrt{3})}{3p},
\end{align*} and both the solutions for $s=0$ are admissible.

\section{ The $S^3\times S^3$ evolution equations}\label{sec:evol-eqnsS3}
 We can now describe the flow equations for the nearly $\G2$-structure on $S^3\times S^3\times \R$ compatible with \eqref{g2-eqns} in this matrix framework. From  Proposition \ref{prop:su3tog2} we know that an invariant nearly half-flat structure $(\omega,\Omega)\in \mathcal{H}_{\lambda,a,b}$. on $S^3\times S^3$ evolve to a nearly parallel $\G2$-structure if and only if it satisfies \eqref{eqns:evolution}.

In terms of matrix $P,Q$ used to parameterize invariant nearly half-flat $\SU(3)$-structures on $S^3\times S^3$ the evolution equations take the following form where $Q_i$s and $A,B,R_i$s are defined in \eqref{eqns:Q1Q2} and \eqref{eqns:ABR1R2} respectively.

\begin{prop}
    The evolution equations for the flow $t\mapsto (P(t),Q(t))\in \mathcal{H}_{\lambda.a(t),b(t)}$ are given by
    \begin{align}
    \begin{split}\label{eqns:evolutionPQ}
        a'&=-\frac{2\lambda}{\det P} A, \ \ \ b'=-\frac{2\lambda}{\det P} B\\
        Q_1'&=-\frac{2\lambda}{\det P} R_1 +P\\
        Q_2'&=-\frac{2\lambda}{\det P} R_2 -P
    \end{split}     
    \end{align}
\end{prop}
\begin{rem}
   In the half-flat case that is when $\lambda=0$ the parameters $a,b$ are constant in $t$ but in the nearly half-flat case the cohomology class $(a,b)$ evolves with time. 
\end{rem}


\subsection{Dynamic examples}
Below we use the the matrix framework to describe some examples of nearly parallel $\G2$-structures on $S^3\times S^3\times I$ for $I\subset \R$ parameterised by $t$.  

\medskip

\noindent
\textbf{The homogeneous nearly $\G2$ metric on the Berger space:}
Let $B\coloneqq \frac{\rm{SO}(5)}{\rm{SO}(3)}$ be the Berger space. The homogeneous metric on $B$ has a nearly parallel $\G2$-structure. There is a cohomogeneity-one action of $\SO(4)$ on $B$ as first described by \cite{Verdiani-Podesta}. Under this action the principal orbits are hypersurfaces of $B$ isomorphic to $\rm{SO}(4)/\Z_2^2\cong \frac{S^3\times S^3}{\Z_2^3}$. The Lie group $\mathrm{SO}(3)$ is embedded into $\mathrm{SO}(5)$ via the $5$ dimensional irreducible representation of $\mathrm{SO}(3)$ on $\mathrm{Sym}^2_0(\R^3)$. If we denote by $S_{i,j}$ the symmetric $3\times 3$ matrix with 1 at the $(i,j)$ and $(j,i)$ entry and $0$ elsewhere then \begin{align*}
    E_1&:=\frac{\rm{diag}(1,1,-2)}{\sqrt{6}}, \quad
     E_2:=\frac{\rm{diag}(1,-1,0)}{\sqrt{2}}, \quad
    E_3:=S_{12}, \quad E_4:=S_{13},\quad E_5:=S_{23}
\end{align*} defines a basis of $\mathrm{Sym}^2_0(\R^3)\cong \R^5$. The embedding of $\rm{SO}(3)$ in $\SO(5)$ by the conjugate action of $\SO(3)$ on $\mathrm{Sym}^2_0(\R^3)\cong \R^5$.

The group $\SO(5)$ acts on $\R^5$ via the usual left multiplication. We can define the group $\rm{SO}(4)=\rm{SO}(4)_{E_1}\subset \rm{SO}(5)$ as the subgroup preserving the $E_1$ direction in $\R^5$. Thus there is an action of $\SO(4)$ on $\R^5$. The generic stabilizer group for the group $\SO(4)_{E_1}$ is also given by $\Z_2^2\cong \rm{diag}(1,1,ab,b,a)$ that preserves the $E_1,E_2$ directions in $\R^5$. The generic orbit is therefore given by $\rm{SO}(4)/\Z_2^2$. The stabilizer of the identity coset $x_-=\id.\rm{SO}(3)$ in $B$ is the group $K^-\cong \rm{O}(2)$ such that $K_0^- \cong \rm{SO}(2)$ acts by angle $2\theta$ in the $E_2,E_3$ plane and by angle $\theta$ in the $E_4,E_5$ plane. Thus $x_-.\rm{SO}(4)=\rm{SO}(4)/(\rm{SO}(2)\times \Z_2)$ is a singular orbit for the action. For the other singular orbit, we need to follow the geodesic $\gamma(t):=\cos(t)E_1+\sin(t)E_2$  transverse to all orbits. At $t=\pi/3$ we get the second singular orbit again isomorphic to $\rm{SO}(4)/(\rm{SO}(2)\times \Z_2)$. Thus the singular stabilizer groups for the action are both isomorphic to $S^1\times\Z_2$.  By some simple calculations one can compute that 
\begin{align*}
    \rm{Stab}(\gamma(t))&\cong\begin{cases}
    \mathbb{Z}_2\times\mathbb{Z}_2 & t\in(0,\pi/3),\\
    S(\mathrm O(2)\mathrm{O}(1)) & t=0, \\
     S(\mathrm O(1)\mathrm{O}(2)) & t=\pi/3.
    \end{cases}
    \end{align*}

With respect to the basis $\{e^1,e^2,e^3,e^4,e^5,e^6,dt\}$ of $\rm{SO}(4)/Z_2^2$ where $e^i$s are as described in \textsection\ref{sec:setup_half_flat} the nearly parallel $\G2$-structure satisfying $d\g2=\frac{6}{\sqrt{5}*_\g2\g2}$ is given by
\begin{align*}
    \g2=& \frac{1}{\sqrt{5}}(\sin(t)e^{12}+\sin(t-2\pi/3)e^{34}+\sin(t+2\pi/3) e^{56})\wedge dt -\frac{-7+2\cos(3t)}{20\sqrt{5}}e^{135}-\frac{7+2\cos(3t)}{20\sqrt{5}} e^{246}\\
    &+\frac{1}{5\sqrt{5}}\left(\cos(t)(e^{235}-e^{146})-3\sin(t-2\pi/3)\sin(t+2\pi/3)(e^{235}+e^{146})\right)\\
     &+\frac{1}{5\sqrt{5}}\left(\cos(t-2\pi/3)(e^{145}-e^{236})-3\sin(t)\sin(t+2\pi/3)(e^{145}+e^{236})\right)\\
      &+\frac{1}{5\sqrt{5}}\left(\cos(t+2\pi/3)(e^{136}-e^{245})-3\sin(t)\sin(t-2\pi/3)(e^{136}+e^{245})\right).
\end{align*}
In the matrix framework the nearly half-flat $\SU(3)$-structure on $\rm{SO}(4)/Z_2^2$ corresponding to the homogeneous nearly parallel $\G2$-structure for $\lambda=6/\sqrt{5}$ on $B$ can then be expressed as 
\begin{align*}
    a&=-\frac{-7+2\cos(3t)}{20\sqrt{5}}, \ \ \ b=-\frac{7+2\cos(3t)}{20\sqrt{5}}, \\
    P&=\frac{1}{\sqrt{5}}\rm{diag}\left(\sin(t),\sin(t-2\pi/3),\sin(t+2\pi/3\right),\\
    Q&=\frac{1}{5\sqrt{5}}\rm{diag}\left(\cos(t),\cos(t-2\pi/3),\cos(t+2\pi/3)\right).
\end{align*} 

\medskip

\noindent
\textbf{Sine-cone over $S^3\times S^3$:} If we assume $P=p(t)\rm{Id}, Q=q(t)\rm{Id}$ with $b(t)=a(t)$ and choose $\lambda=4$. The normalization condition takes the following form 
\[
48p^8(t) +(64a(t)-1)p^6(t)+24p^4(t)(a^2(t)-q^2(t))+48p^2(t)q^2(t)a(t)+3q^4(t)-6q^2(t)a^2(t)-a^4(t)=0
\]

Setting the coefficient of $e^{135}-e^{246}$ to zero in  $\gamma'-d\omega+4J\gamma$ we get 
\begin{align*}
    -\frac{16(a(t)-4p^2(t))((a(t)+2p^2(t))^2+3q^2(t))}{p^3(t)}=0
\end{align*}
which implies either $a(t)=4p^2(t)$ or if $(a(t)+2p^2(t))^2+3q^2(t)=0$ but the later solution solves the normalization condition if and only if $p(t)=0$ and can be discarded. Substituting $a(t)=4p^2(t)$ in the normalization condition generates four possible solutions 
\begin{align*}
    q(t)&=\pm\frac{p(t)\sqrt{-108p(t)^2\pm 3\sqrt{3}p(t)}}{3}.
\end{align*}
Equating all the coefficients in $\gamma'-d\omega+4J\gamma$ to zero give the following set of ODEs 
\begin{align*}
    p'(t)p^4(t)&=24p^4(t)q(t)+2q^3(t),\\
    q'(t)p(t)&=p^2(t)-48q^2(t)-576p^4(t).
    \end{align*}
Substituting $q(t)$ in terms of $p(t)$ in the above equations we get either 
\begin{align*}
    p(t)&=\pm\frac{\sqrt{3}}{72}(1+\sin(4t+c))\\
    q(t)&=\pm\frac{\sqrt{3}}{864}|\cos(4t+c)|(1+\sin(4t+c)),
\end{align*}
or 
\begin{align*}
    p(t)&=\pm\frac{\sqrt{3}}{72}(-1+\sin(4t+c))\\
    q(t)&=\pm\frac{\sqrt{3}}{864}|\cos(4t+c)|(-1+\sin(4t+c)).
\end{align*}

If we now assume the nearly half-flat structure at $t=0$ to be the unique nearly K\"ahler solution presented in  \eqref{eqn:nk-soln} we get that the following solutions to the evolution equations
\begin{align*}
a(t)=b(t)&=\frac{\cos^4(2t)}{108},\ \ \ 
    p(t)=\frac{\sqrt{3}}{36}\cos^2(2t),\ \ \
    q(t)=\frac{\sqrt{3}}{216}\cos^3(2t)\sin(2t),
\end{align*}
or
\begin{align*}
a(t)=b(t)&=\frac{\cos^4(2t)}{108},\ \ \ 
    p(t)=-\frac{\sqrt{3}}{36}\cos^2(2t),\ \ \ 
    q(t)=-\frac{\sqrt{3}}{216}\cos^3(2t)\sin(2t).
\end{align*}

If we denote by $g_{NK}$, the metric induced by the nearly K\"ahler $\SU(3)$-structure, the metric $g_6(t)$ corresponding to the above nearly half-flat structure at any time $t$ is given by
\begin{align*}
    g_6(t)&= \frac{\cos^2(2t)}{18}\sum_{i=1}^3((e^{2i-1})^2+(e^{2i})^2-\frac{1}{2}e^{2i-1}\otimes e^{2i}-\frac{1}{2}e^{2i}\otimes e^{2i-1})= \cos^2(2t)g_{NK}.
\end{align*}

From \eqref{g2-eqns} the nearly $\G2$-structure is given by
\begin{align*}
    \g2=&\frac{\sqrt{3}\cos^2(2t)}{36}(e^{12}+e^{34}+e^{56})\wedge dt +\frac{\cos^4(2t)}{108}(e^{135}+e^{246})\\
   &+ \frac{\cos^3(2t)\cos(2t-2\pi/3)}{108}(e^{136}+e^{145}+e^{235})+\frac{\cos^3(2t)\cos(2t+2\pi/3)}{108}(e^{146}+e^{236}+e^{245})
\end{align*}

Reparametrizing $s=2t+\pi/2$, the $\G2$-metric $g_\g2(s)$ corresponding to the above $\G2$-structure is (up-to a scale) given by
\begin{align*}
    g_\g2(s)&= \sin^2(s)g_{NK}+(ds)^2,
\end{align*} which is the well-known sine-cone nearly $\G2$ metric over the nearly K\"ahler $\SU(3)$-structure on $S^3\times S^3$ (see \cite{Acharya:2003ii}).

Note that the above solution is incomplete as at $s=0,\pi$ the metric shrinks to a point and becomes highly non-singular. At any time $s\in (0,\pi)$ the only non-vanishing torsion form for the nearly half-flat structure is $w_1^+=6\cot(s)$ which vanishes only at $s=\pi/2$. 

\medskip

One can also explicitly write down other known examples of nearly parallel $\G2$-structures on $S^3\times S^3\times I$ such as the homogeneous nearly parallel $\G2$-structure on $S^7$. The Lie group SO(4) acts by cohomogeneity-one on $S^7 \subset \R^8$. The action of $\SO(4)$ on $\R^8$ is defined by the isotropy action on the tangent space of $\G2/\SO(4)$. As a complex representation of $\SU(2) \times \SU(2)$ if $V_{(k,l)}$ represents the tensor product of the symmetric representation of weight $k, l$ on the first and second $\SU(2)$ factor
respectively, the space $\R^8$ can be written as $V_{(1,0)} \otimes V_{(0,3)}$. Moreover in \cite{wilking-3sasaki} the authors showed that the total spaces of the  $\SO(3)$-bundles of self dual (anti-self dual) $2$-forms over Hitchin's self dual Einstein orbifolds \cite{Hitchin_Einstein} are smooth $3$-Sasakian seven dimensional manifolds. The action of $\SO(3)$ on the base lifts to form a cohomogeneity-one $\SO(3)\times  \SO(3)$-action on the total space. One can describe the 2-parameter family of nearly parallel $\G2$-structures induced by this $3$-Sasakian structure using the matrix framework. 

Apart from this, one can also use this framework to study $\G2$-instantons on $S^3\times S^3\times I$ with respect to the $\G2$-structure $(\g2,\psi)$ defined in \eqref{eq:ng2defn}. In \cite[Lemma 1]{Lotay2018-G_2instantonsON} the authors showed that if $\psi$ is closed then $\G2$-instantons on $S^3\times S^3\times I$ are in one-to-one correspondence with a 1-parameter family of connections $a(t)_{t\in I}$ with curvature $F_a(t)$ on $S^3\times S^3$  that satisfies
\begin{align*}
    \dot{a}\wedge\frac{\omega^2}{2}-F_a\wedge J\gamma &=0, 
\end{align*}
along with the constraint $F_a\wedge \omega^2=0$, which is shown to be compatible with the evolution. Previously in \cite{lotay-oliveira} the authors used similar framework to describe the $\SU(2)^2$-invariant $\G2$-instantons on non-compact manifolds with holonomy $\G2$. A similar analysis can be done when the $\G2$-structure is nearly parallel.

\phantomsection
\addcontentsline{toc}{section}{References}

\medskip

\bibliographystyle{amsalpha}
\bibliography{main}

\vspace{0.8cm}

\noindent 
Department of Mathematics, Unversit\'e Libre de Bruxelles,  Boulevard du Triomphe 155. B-1050 Bruxelles.\\
\emph{E-mail address} : \href{mailto:ragini.singhal@ulb.be}{ragini.singhal@ulb.be}, \href{mailto:raginisinghal1016@gmail.com}{raginisinghal1016@gmail.com}
\end{document}